\documentclass[reqno]{amsart}

\usepackage{amsmath}
\usepackage{amssymb}
\usepackage{amsthm}
\usepackage{url}

\DeclareMathOperator{\Ric}{Ric}
\DeclareMathOperator{\Hess}{Hess}
\DeclareMathOperator{\tr}{tr}
\DeclareMathOperator{\edge}{edge}
\DeclareMathOperator{\sech}{sech}
\DeclareMathOperator{\divergence}{div}

\newcommand{\lp}{\langle}
\newcommand{\rp}{\rangle}

\newtheorem{thm}{Theorem}[section]
\newtheorem{prop}[thm]{Proposition}
\newtheorem{cor}[thm]{Corollary}
\newtheorem{lem}[thm]{Lemma}
\newtheorem{conj}[thm]{Conjecture}

\theoremstyle{definition}
\newtheorem{defn}[thm]{Definition}
\newtheorem{remark}[thm]{Remark}
\newtheorem{ex}[thm]{Example}

\theoremstyle{remark}

\numberwithin{equation}{section}

\begin{document}

\title[Singularity Theorems and the Lorentzian Splitting Theorem]{Singularity theorems and the Lorentzian splitting theorem for the Bakry-Emery-Ricci tensor}
\author{Jeffrey S. Case}
\address{Department of Mathematics\\
University of California\\
Santa Barbara, CA 93106}
\email{casej@math.ucsb.edu}
\date{}
\begin{abstract}
We consider the Hawking-Penrose singularity theorems and the Lorentzian splitting theorem under the weaker curvature condition of nonnegative Bakry-Emery-Ricci curvature $\Ric_f^m$ in timelike directions.  We prove that they still hold when $m$ is finite, and when $m$ is infinite, they hold under the additional assumption that $f$ is bounded from above.
\end{abstract}
\maketitle

\section{Introduction}

Recently, the Bakry-Emery-Ricci tensor $\Ric_f$ has become an important object of study in Riemannian geometry, in large part due to its appearance in the study of Ricci flow and Ricci solitons.  The Bakry-Emery-Ricci tensor is defined by
\[ \Ric_f = \Ric + \Hess f, \]
where $f$ is a smooth function on $M$.  The Bakry-Emery-Ricci tensor also arises naturally in the study of metric measure spaces $(M,g,e^{-f}dvol_g)$, where $(M,g)$ is a Riemannian manifold, $f$ is a smooth function on $M$, and $dvol_g$ is the Riemannian volume density on $M$.  For a brief overview of the results in Riemannian geometry and references to results in both situations, we refer the reader to \cite{WW}.

In this paper, we consider weakening the timelike convergence condition on a Lorentzian manifold by introducing the $m$-Bakry-Emery-Ricci curvature $\Ric_f^m$.  If $m>0$ and $f\colon M\to\mathbb{R}$ is a smooth function, we define
\[ \Ric_f^m = \Ric + \Hess f - \frac{1}{m}df\otimes df. \]
In the case that $m$ is a positive integer, $H$ is a Lorentzian manifold, and $M$ is a pseudo-Riemannian manifold, this condition arises in considering the Ricci curvature of the warped product $H\times_\phi M$, with $m=\dim M$ and
\[ \phi=-\frac{1}{m}\ln f. \]
For a definition and discussion of the basic properties of warped products, see \cite{O}.  In \cite{L}, Lott studied the Bakry-Emery-Ricci tensor for a Riemannian manifold $M$ by considering warped products of the form $M\times S^N(r)$ for large $N$, where $S^N(r)$ is the standard $N$-sphere of radius $r$, by noting that the condition $\Ric_f^N\geq \lambda$ implies that $\overline{\Ric}\geq \lambda$, with $\overline{\Ric}$ the Ricci curvature of the warped product.  By letting $r\to 0$, Lott was able to show that many topological results for Ricci curvature bounded below extend to Bakry-Emery-Ricci curvature bounded below.  While not considered in this paper, his technique is analogous to the technique of dimensional reduction in physics, and it is likely that the results of this paper would be of use there.  The author would like to thank Xianzhe Dai for pointing this out to him.

When $f$ is constant, the Bakry-Emery-Ricci curvature is the Ricci curvature, so we ask if we can extend results about manifolds which satisfy the timelike curvature condition $\Ric(v,v)\geq 0$ for timelike vectors $v$ to the Bakry-Emery-Ricci-Curvature.  Specifically, we will consider singularity theorems of the type proven by Hawking and Penrose in \cite{HP}, and the Lorentzian splitting theorem, as proven in \cite{B+}, \cite{E}, \cite{G}, and \cite{N}.  The primary tools in proving these theorems is through the Raychaudhuri equation, which is related to the Riccati equation, and the maximum principle for spacelike hypersurfaces, respectively.  We establish these tools under the assumption of non-negative Bakry-Emery-Ricci tensor along timelike vectors by considering a modification of the Laplacian by $\Delta_f =\Delta - \nabla f\cdot\nabla,$ analogous to \cite{FLZ} and \cite{WW}.

Just as in the Riemannian case, when dealing with the $\infty$-Bakry-Emery-Ricci tensor $\Ric_f=\Ric_f^\infty$, the results do not extend for arbitrary functions $f$.  In Section \ref{example_section}, we give an examples where the results do not extend.  However, if we assume additionally that the function $f$ is bounded above, the results still extend.

In this paper, we prove Hawking-Penrose singularity theorems under three assumptions.  Specifically,

\begin{thm}[Singularity Theorem] \label{singularity} Let $M$ be a chronological space-time of dimension $n\geq 3$ which satisfies the $f$-generic condition and the $(m,f)$-timelike convergence condition.  If $m=\infty$, assume additionally that $f$ is bounded above.  Then the space-time $M$ is nonspacelike incomplete if any of the following three conditions hold:
\begin{itemize}
\item[(1)] $M$ has a closed $f$-trapped surface.
\item[(2)] $M$ has a point $p$ such that each null geodesic starting at $p$ is $f$-reconverging somewhere in the future \textup{(}or past\textup{)} of $p$.
\item[(3)] $M$ has a compact spacelike hypersurface.
\end{itemize}
\end{thm}

Here, the $(m,f)$-timelike curvature condition refers to the condition that $\Ric_f^m\geq 0$ for timelike vectors, and the $f$-generic condition is a generalization of the generic condition to a modification of the curvature endomorphism $R(\cdot,c^\prime)c^\prime)$ along a nonspacelike geodesic $c$.  For the details, see Definition \ref{curv_end}.  The key idea in establishing the above generalization is proving the existence of conjugate points along any complete nonspacelike geodesic under the given curvature conditions, given in \ref{nonspace_conjpts}.

After establishing the singularity theorems, we then prove the Lorentzian splitting theorem.

\begin{thm}[Splitting Theorem] \label{splitting} Let $M$ be a connected space-time satisfying the following conditions:
\begin{itemize}
\item[(1)] $M$ is timelike geodesically complete or $M$ is globally hyperbolic,
\item[(2)] $M$ contains a timelike line, and
\item[(3)] $\Ric_f^m(v,v)\geq 0$ for all timelike vectors $v$.  If $m=\infty$, assume additionally that $f$ is bounded above.
\end{itemize}
Then $M$ is isometric to $(\mathbb{R}\times S, -dt^2\oplus h)$, where $(S,h)$ is a complete Riemannian manifold, and $f$ is constant along $\mathbb{R}$.
\end{thm}

We will give a sketch of the proof, following that given in \cite{GH}, which follows after establishing a maximum principle for spacelike hypersurfaces in Theorem \ref{max_principle}.  For details of the proof under the normal Ricci curvature assumption in the case that the manifold is globally hyperbolic or timelike geodesically complete, we refer the reader to \cite{G} or \cite{N}, respectively.

The author would like to thank Xianzhe Dai and Guofang Wei for their interest and helpful discussions.

\section{The Raychaudhuri Equation}

We will always denote by $M$ a Lorentzian space-time with metric $g$ and dimension $n$.  We then adopt the conventions of \cite{BEE} for defining the causal structures $I^{\pm}$, $J^{\pm}$, $D^{\pm}$, and $E^{\pm}$, and refer the reader there for basic definitions on Lorentzian space-times.

\begin{defn} Let $f\colon M\to\mathbb{R}$ be a smooth function and let $m$ be a positive integer.  The $m$-Bakry-Emery-Ricci tensor $\Ric_f^m$ is defined by
\[ \Ric_f^m = \Ric + \Hess f - \frac{1}{m}df\otimes df, \]
where $\Hess f$ is the Hessian of $f$ as a (0,2) tensor. The Bakry-Emery-Ricci tensor $\Ric_f$ is
\[ \Ric_f = \Ric_f^\infty = \Ric + \Hess f. \] \end{defn}

We first remind the reader of the definitions of Jacobi and Lagrange tensors, which are of great use in the study of conjugate points.

\begin{defn} Let $c\colon[a,b]\to M$ be a future-directed timelike unit speed geodesic.  Denote by $N(c(t))=(c^\prime(t))^\perp$ the $(n-1)$ subspace orthogonal to $c^\prime(t)$ in $T_{c(t)}M$.  Then tensor field $A=A(t)\colon N(c(t))\to N(c(t))$ is a Jacobi tensor field if
\[ A^{\prime\prime} + RA = 0 \qquad\mbox{and}\qquad \ker(A(t))\cap\ker(A^\prime(t))=\{0\}\mbox{ for all } t\in[a,b], \]
where $R=R(t)\colon N(c(t))\to N(c(t))$ denotes the curvature endomorphism
\[ R(t)(v) = R(v, c^\prime(t))c^\prime(t), \]
where $R$ is the usual $(1,3)$ Riemannian curvature tensor on $M$. \end{defn}

\begin{defn} A Jacobi tensor field $A$ is a Lagrange tensor field if
\[ (A^\prime)^\ast A - A^\ast A^\prime = 0. \]
\end{defn}

\begin{prop} \label{lagrange_prop} Let $A$ be a Jacobi tensor field along $c\colon[a,b]\to M$ and suppose that $A(s)=0$ for some $s\in[a,b]$.  Then $A$ is a Lagrange tensor. \end{prop}

\begin{proof} 

First note that, since $A$ is a Jacobi tensor field, $(A^\ast)^\prime = (A^\prime)^\ast$, and $R$ is self-adjoint,
\[ ((A^\prime)^\ast A - A^\ast A^\prime)^\prime = 0. \]
Thus the quantity $(A^\prime)^\ast A - A^\ast A^\prime=k$ for some constant $k$.  Since $A(s)=0$, we have that $c=0$.
\end{proof}

\begin{remark} \label{jacobi} If $Y$ is a parallel vector field along $c$ and if $A$ is a Jacobi tensor field, then $A(Y)$ is a Jacobi field along $c$ in the ordinary sense.  Moreover, if we let $E_1,\dotsc,E_{n-1}$ be an orthonormal frame field along $c$, and we let $J_i$ be the unique Jacobi field along $c$ such that $J_i(a)=0$ and $J_i^\prime(0)=E_i$, then defining $A$ by the matrix $A=[J_1 J_2 \dotso J_{n-1}]$, where each column is just the vector for $J_i$ in the basis defined by $\{E_i\}$, one easily checks that $A$ is the matrix representation of a Lagrange tensor field.  In this situation, we have that $A(t)$ is invertible if and only if $c(t)$ is not conjugate to $c(a)$, and this observation will be used frequently in establishing the singularity theorems.  For details, see \cite{BEE}. \end{remark}

\begin{defn} Let $A$ be a Jacobi tensor field along $c$ and let $f\colon M\to R$ be a smooth function, and define $B_f = A^\prime A^{-1} - \frac{1}{n-1}(f\circ c)^\prime E$ wherever $A$ is invertible, where $E(t)$ is the identity map on $N(c(t))$.  Then we define
\begin{itemize}
\item[(1)] the $f$-expansion function $\theta_f$ by
\[ \theta_f = \tr B_f, \]
\item[(2)] the $f$-vorticity tensor $\omega_f$ by
\[ \omega_f = \frac{1}{2}(B_f-B_f^\ast), \]
\item[(3)] the $f$-shear tensor $\sigma_f$ by
\[ \sigma_f = \frac{1}{2}(B_f+B_f^\ast) - \frac{\theta_f}{n-1}E. \]
\end{itemize}
\end{defn}

Note that if $f$ is constant, then $f^\prime\equiv 0$ and we recover the ordinary expansion, vorticity tensor, and shear tensor of a Jacobi tensor field, as defined in \cite{BEE}.  Furthermore, since the identity map is self-adjoint, we immediately have that $\omega_f = \omega$ and $\sigma_f=\sigma$.

It is a basic result of linear algebra that $\tr (A^\prime A^{-1}) = (\det A)^{-1}(\det A)^\prime$, so by defining $A$ as in Remark \ref{jacobi}, if we know that $|\theta|\to\infty$ as $t\to t_1$ for some $t_1\in[a,b]$, where $\theta=\tr(A^\prime A^{-1})$, then we know that $\det A(t_1)=0$, and hence $c(t_1)$ is conjugate to $c(a)$.  Since $(f\circ c)^\prime$ is finite wherever $c$ is defined, the same holds true for $\theta_f$.  Also, $\det A$ is exactly the volume element, and so $\theta$ corresponds to the mean curvature of the distance ball $B(c(a), t)$.  Thus, if we compare with \cite{WW}, we see that we are just redefining the notion of the rescaled mean curvature $m_f=m-\partial_r f$ to the appropriate Lorentzian framework.

\begin{defn} \label{curv_end} Let $c\colon J\to M$ be a timelike geodesic and let $f$ be a smooth function on $M$.  Then $R_f\colon N(c(t))\to N(c(t))$ is defined by
\[ R_f(t) = R(t) + \frac{1}{n-1}\Hess f(c^\prime, c^\prime) E + \left(\frac{1}{n-1}\right)^2 ((f\circ c)^\prime)^2 E, \]
where $E\colon N(c(t))\to N(c(t))$ is the identity map. \end{defn}

\begin{prop} Let $c\colon J\to M$ be a timelike geodesic and let $f$ be a smooth function on $M$.  If $A$ is a Jacobi tensor field, then
\[ R_f = -B_f^\prime - B_f^2 - \frac{2}{n-1}(f\circ c)^\prime B_f. \]
\end{prop}

\begin{proof}

From \cite{BEE}, we have the equation
\[ R = -B^\prime - B^2. \]
Using that $B=B_f + \frac{1}{n-1}(f\circ c)^\prime E$, a simple calculation yields the result.
\end{proof}

\begin{prop} If $A$ is a Jacobi tensor field, $f$ is a smooth function on $M$, and $m$ is a positive integer or $m=\infty$, then
\[ \theta_f^\prime = -\Ric_f^m(c^\prime,c^\prime)-\tr\omega_f^2 - \tr\sigma_f^2 - \frac{\theta^2}{n-1} - \frac{((f\circ c)^\prime)^2}{m}. \]
We will call this the $(m,f)$-Raychaudhuri equation.  In the case that $m<\infty$, this yields the inequality
\[ \theta_f^\prime \leq -\Ric_f^m(c^\prime, c^\prime) - \tr\omega_f^2 - \tr\sigma_f^2 - \frac{\theta_f^2}{n+m-1}, \]
and in the case that $m=\infty$,
\[ \theta_f^\prime \leq -\Ric_f(c^\prime,c^\prime) - \tr\omega_f^2 - \tr\sigma_f^2 - \frac{\theta_f^2}{n-1} - \frac{2\theta_f}{n-1}(f\circ c)^\prime. \]
\end{prop}

\begin{proof}

First recall the ordinary Raychaudhuri equation
\[ \theta^\prime = -\Ric(c^\prime,c^\prime) - \tr\omega^2 - \tr\sigma^2 - \frac{\theta^2}{n-1}. \]
Now, noting that
\[ \theta_f^\prime = \theta^\prime - f^{\prime\prime} = \theta^\prime - \Hess f(c^\prime,c^\prime) \]
and using the definition of $\Ric_f^m$ as well as the facts that $\omega_f=\omega$ and $\sigma_f=\sigma$, we get
\[ \theta_f^\prime = -\Ric_f^m(c^\prime,c^\prime) - \tr\omega_f^2 - \tr\sigma_f^2 - \frac{\theta^2}{n-1} - \frac{((f\circ c)^\prime)^2}{m} \]

In the case that $m<\infty$, using the inequality
\begin{equation} \label{schwarz} \frac{\theta^2}{n-1} + \frac{((f\circ c)^\prime)^2}{m} \geq \frac{(\theta \pm (f\circ c)^\prime)^2}{n+m-1}, \end{equation}
we get
\[ \theta_f^\prime \leq -\Ric_f^m(c^\prime,c^\prime) - \tr\omega_f^2 - \tr\sigma_f^2 - \frac{\theta_f^2}{n+m-1}. \]
We note that equality holds if and only if $\theta = \pm\sqrt{\frac{n-1}{m}}(f\circ c)^\prime$, and so is quite rare in our situation.

On the other hand, if $m=\infty$, using that $\theta=\theta_f+(f\circ c)^\prime$, then
\begin{align*}
\theta_f^\prime & = -\Ric_f^m(c^\prime,c^\prime) - \tr\omega_f^2 - \tr\sigma_f^2 - \frac{(\theta+(f\circ c)^\prime)^2}{n-1} \\
& \leq -\Ric_f^m(c^\prime,c^\prime) - \tr\omega_f^2 - \tr\sigma_f^2 - \frac{\theta_f^2}{n-1} - \frac{2\theta_f}{n-1}(f\circ c)^\prime. \qedhere
\end{align*}
\end{proof}

\begin{cor} If $A$ is constructed from Jacobi fields as described above, then
\[ \theta_f^\prime = -\Ric_f^m(c^\prime,c^\prime) - \tr\sigma_f^2 - \frac{\theta^2}{n-1} - \frac{((f\circ c)^\prime)^2}{m}. \]
This we call the vorticity-free $(m,f)$-Raychaudhuri equation.  This yields the inequalities
\begin{equation} \label{finite_raych} \theta_f^\prime\leq -\Ric_f^m(c^\prime,c^\prime) - \tr\sigma_f^2 - \frac{\theta_f^2}{n+m-1} \end{equation}
\begin{equation} \label{infinite_raych} \theta_f^\prime\leq -\Ric_f(c^\prime,c^\prime) - \tr\sigma_f^2 - \frac{\theta_f^2}{n-1} - \frac{2\theta_f}{n-1}(f\circ c)^\prime \end{equation}
in the respective cases of $m<\infty$ and $m=\infty$. \end{cor}

\begin{proof} Since $A$ is actually a Lagrange tensor field, we quickly check that $B_f$ is self-adjoint, and hence $\omega_f\equiv 0$. \end{proof}

\begin{remark} \label{selfadjoint} First, if we assume $f$ is constant and $m=0$, then we recover the original vorticity-free Raychaudhuri equation.  Also, we note that by definition, $\sigma_f$ is self-adjoint, and hence $\tr\sigma_f^2\geq 0$, with equality if and only if $\sigma\equiv 0$.  Hence, as we shall see in the next section, assuming that $\Ric_f^m(c^\prime,c^\prime)\geq 0$ and that $A$ is a Lagrange tensor field, the vorticity-free $(m,f)$-Raychaudhuri equation reduces to
\[ \theta_f^\prime \leq -\frac{1}{n-1}\theta_f^2, \]
which will impose conditions on $\theta_f$ that we can exploit. \end{remark}

In a straightforward manner, we can extend our definitions to null geodesics $\beta$, as in \cite{BEE}.  To do so, we must restrict our tensors to maps on the quotient bundle $G(\beta)=N(\beta(t))/[\beta^\prime(t)]$, an $(n-2)$-dimensional subspace of $T_{\beta(t)}M$.  Without going to the quotient bundle, the fact that $g(\beta^\prime,\beta^\prime)\equiv 0$ will force $A$ to be degenerate always, but by going to the quotient bundle, we return to the situation where $A$ is degenerate only at points conjugate to $\beta(a)$.  We then denote by $\overline{R}$ the curvature endomorphism defined on $G(\beta)$, $\overline{\theta_f}$ the $f$-expansion for a Jacobi tensor field $\overline{A}$ defined on $G(\beta)$, and so on.

\section{Existence of Conjugate Points}

A key argument in proving singularity theorems is showing the existence of conjugate points along geodesics.  Following \cite[Chapter 12]{BEE}, we show in this section that conjugate points always exist along nonspacelike complete geodesics as long as certain curvature conditions are satisfied.  One is the $(m,f)$-timelike convergence condition, and the other is an appropriate generalization of the generic condition of relativity.

\begin{defn} A timelike geodesic $c\colon[a,b]\to M$ satisfies the $f$-generic condition if there exists some $t_0\in(a,b)$ such that $R_f(t_0)\not= 0$.  A null geodesic satisfies the $f$-generic condition if the same holds true for $\overline{R_f}$.  The space-time $(M,g)$ satisfies the $f$-generic condition if each inextendible nonspacelike geodesic satisfies the $f$-generic condition.
\end{defn}

The singularity theorems we prove will assume that $\Ric_f^m\geq 0$ for timelike vectors and that the $f$-generic condition holds.  If we take the trace of $R_f$, we get
\begin{align*}
\tr R_f & = \Ric(c^\prime, c^\prime)+\Hess f(c^\prime,c^\prime) + \frac{1}{n-1}((f\circ c)^\prime)^2 \\
& = \Ric_f^m(c^\prime,c^\prime) + \left(\frac{1}{n-1}+\frac{1}{m}\right)((f\circ c)^\prime)^2.
\end{align*}
Hence, if we assume that $\Ric_f^m > 0$ along timelike vectors, then automatically $R_f\not\equiv 0$, and so the $f$-generic condition is satisfied.  In this manner, we see how the splitting theorem can be viewed as a rigidity result for the singularity theorems.

As mentioned in the previous section, one way to prove the existence of conjugate points along a nonspacelike geodesic is by showing that the expansion $\theta_f(t)$ must blow up in finite time.  The following two propositions give us the means to do that:

\begin{prop} \label{raych} Let $c\colon J\to M$ be an inextendible timelike geodesic satisfying $\Ric_f^m(c^\prime(t),c^\prime(t))\geq 0$ for all $t\in J$, with $m<\infty$.  Let $A$ be a Lagrange tensor field along $c$, and $\theta_f$ be the $f$-expansion.  Suppose that $\theta_1=\theta_f(t_1)$ is positive \textup{(}resp., negative\textup{)} at $t_1\in J$.  Then $\det A(t)=0$ for some $t\in [t_1-(n+m-1)/\theta_1, t_1]$ \textup{(}resp., $t\in [t_1, t_1-(n+m-1)/\theta_1]$\textup{)} provided that $t\in J$. \end{prop}

\begin{proof}

As mentioned above, we need only show that $|\theta|\to\infty$ in the above intervals.  Define
\[ s_1 = \frac{n+m-1}{\theta_1}. \]
From the vorticity-free $(m,f)$-Raychaudhuri equation for timelike geodesics and the condition $\Ric_f^m(c^\prime,c^\prime)\geq 0$, we have that
\[ \frac{d\theta_f}{dt} \leq -\frac{\theta_f^2}{n+m-1}. \]
Assuming $\theta_1 > 0$, integrating this inequality from $t<t_1$ to $t_1$ yields
\[ \theta_f(t) \geq \frac{n+m-1}{t+s_1-t_1} \]
for $t\in (t_1-s_1, t_1]$.  Hence $|\theta(t)|$ becomes infinite for some $t\in [t-s_1, t_1]$ provided $t\in J$.  If instead $\theta_1 < 0$, integrating from $t$ to $t>t_1$ yields
\[ \theta_f(t) \leq \frac{n+m-1}{t+s_1-t_1} \]
for $t\in [t_1,t_1-s_1)$.  Hence we again have that $|\theta(t)|$ becomes infinite for some $t\in [t_1, t_1-s_1]$ provided $t\in J$.
\end{proof}

\begin{prop} \label{raych2} Let $c\colon J\to M$ be a timelike geodesic satisfying $\Ric_f(c^\prime,c^\prime) \geq 0$.  Then if $f\leq k$ and $\theta_f(t_1) > 0$ \textup{(}resp., $\theta_f(t_1) < 0$\textup{)} for some $t_1\in J$, there is a $t\in [t_1-\sigma,t_1]$ \textup{(}resp., $t\in [t_1,t_1-\sigma]$\textup{)} such that $c(t)$ is conjugate to $c(t_1)$, provided $t\in J$, where
\[ \sigma = \frac{n-1+2k-2f(c(t_1))}{\theta_f(t_1)}. \] \end{prop}

\begin{proof}

From \eqref{infinite_raych} and the vorticity-free $f$-Raychaudhuri equation, we have that $\theta_f^\prime\leq 0$ and
\[ \theta_f^\prime \leq -\frac{\theta_f^2}{n-1}-\frac{2\theta_f}{n-1}(f\circ c)^\prime. \]
Assume that $\theta_f(t_1)<0$, so that $\theta_f(t)\leq\theta_f(t_1)<0$ for $t\geq t_1$.  Multiplying the above inequality by $(n-1)/\theta_f^2$ and integrating from $t_1$ to $t>t_1$, we have that
\begin{align*}
0 < -\frac{n-1}{\theta_f(t)} & \leq t_1-t-\frac{n-1}{\theta_f(t_1)}-2\int_{t_1}^t \frac{(f\circ c)^\prime}{\theta_f} \\
& = t_1-t-\frac{n-1}{\theta_f(s)}-\frac{2f(c(t))}{\theta_f(t)}+\frac{2f(c(t_1))}{\theta_f(t_1)} - 2\int_{t_1}^t \frac{\theta_f^\prime}{\theta_f^2}(f\circ c) \\
& \leq t_1-t-\frac{n-1}{\theta_f(t_1)}-\frac{2f(c(t))}{\theta_f(t)}+\frac{2f(c(t_1))}{\theta_f(t_1)} + 2k\left(\frac{1}{\theta_f(t)}-\frac{1}{\theta_f(t_1)}\right) \\
& \leq t_1-t-\frac{n-1+2k-2f(c(t_1))}{\theta_f(t_1)}.
\end{align*}
This is impossible if we let $t=t_1-\sigma$, and thus there must be a $t\in [t_1,t_1-\sigma]$ such that $|\theta_f(s)|\to \infty$ as $s\to t$, and so $c(t)$ is conjugate to $c(t_1)$.

In the case that $\theta_f(t_1)>0$, the proof is identical except that now we integrate from $t<t_1$ to $t$.
\end{proof}

This leads now to being able to show that a timelike geodesic in a space-time that satisfies $\Ric_f^m\geq 0$ along timelike vectors and the $f$-generic condition must either be incomplete or have a pair of conjugate points.

\begin{prop} \label{conjpts} Let $(M,g)$ be an arbitrary space-time of dimension $n\geq 2$.  Suppose that $c\colon\mathbb{R}\to M$ is a complete timelike geodesic which satisfies $\Ric_f^m(c^\prime,c^\prime)\geq 0$ \textup{(}if $m=\infty$, assume additionally that $f\leq k$\textup{)}.  If $c$ satisfies the $f$-generic condition for some $t_1\in\mathbb{R}$, then $c$ has a pair of conjugate points. \end{prop}

We first note that we assume $n\geq 2$ so that $R_f$ is a nontrivial map, as thus the $f$-generic condition can be satisfied.  For similar reasons, we will need to assume that $n\geq 3$ when considering null geodesics and their conjugate points.  The idea of the proof is to construct an appropriate Lagrange tensor field along $c$ such that the assumption that $c$ satisfies the $f$-generic condition will imply that the expansion is not identically zero, and then use the previous proposition to construct a pair of conjugate points.  To do this, we will need four lemmas, the first three of which are proven in \cite{BEE}.

\begin{lem}[{\cite[Lemma 12.11]{BEE}}] \label{lemma1} Let $c\colon[a,b]\to M$ be a timelike geodesic without conjugate points.  Then there is a unique $(1,1)$ tensor field $A$ on $N(c(t))$ which satisfies the differential equation $A^{\prime\prime}+RA = 0$ with given boundary conditions $A(a)$ and $A(b)$. \end{lem}

Now let $c\colon [t_1,\infty)\to M$ be a timelike geodesic without conjugate points and fix $s\in (t_1,\infty)$.  Then by Lemma \ref{lemma1}, there exists a unique $(1,1)$ tensor field $D_s$ on $N(c(t))$ satisfying the differential equation $D_s^{\prime\prime} + RD_s = 0$ with initial conditions $D_s(t_1)=E$ and $D_s(s)=0$.  Since $D_s(t_1)=E$, $\ker(D_s(t_1))\cap\ker(D_s^\prime(t_1))=\{0\}$, so $D_s$ is a Jacobi tensor field.  Also, since $D_s(s) = 0$, Proposition \ref{lagrange_prop} implies that $D_s$ is a Lagrange tensor field.  This leads us to the next lemma:

\begin{lem}[{\cite[Lemma 12.12]{BEE}}] \label{lemma2} Let $c\colon[t_1,\infty)\to M$ be a timelike geodesic without conjugate points.  Let $A$ be the unique Lagrange tensor on $N(c(t))$ with $A(t_1) = 0$ and $A^\prime(t_1) = E$.  Then for each $s\in (t_1,\infty)$, the Lagrange tensor $D_s$ on $N(c(t))$ with $D_s(t_1)=E$ and $D_s(s) = 0$ satisfies the equation
\[ D_s(t) = A(t)\int_t^s (A^\ast A)^{-1}(\tau)\;d\tau \]
for all $t\in (t_1,s]$.  Thus $D_s(t)$ is nonsingular for $t\in (t_1,s)$. \end{lem}

It is worth noting that $A$ is constructed as in Remark \ref{jacobi}, and its uniqueness follows from the uniqueness of Jacobi fields $J$ under the initial conditions $J(t_0)=v$ and $J^\prime(t_0)=w$ for vectors $v,w\in T_{c(t_0)}M$.  Also, notice that in this lemma, we have proven that $D_s^\prime(s) = -(A^\ast)^{-1}(s)$.

Next, we have that the tensor fields $D_s$ converge to a Lagrange tensor field $D$ as $s\to\infty$.

\begin{lem}[{\cite[Lemma 12.13]{BEE}}] \label{lemma3} Let $c\colon[a,\infty)\to M$ be a timelike geodesic without conjugate points.  For $t_1>a$ and $s\in [a,\infty)-\{t_1\}$, let $D_s$ be the Lagrange tensor field along $c$ determined by $D_s(t_1)=E$ and $D_s(s)=0$.  Then $D(t)=\lim_{s\to\infty} D_s(t)$ is a Lagrange tensor field.  Furthermore, $D(t)$ is nonsingular for all $t$ with $t_1<t<\infty$. \end{lem}

In the proof of Lemma \ref{lemma3}, $D(t)$ is constructed by first noting that $D_s^\prime(t_1)$ has a self-adjoint limit, and then defining $D(t)$ to be the unique Jacobi tensor field satisfying $D(t_1)=E$ and $D^\prime(t_1)=\lim_{s\to\infty} D_s^\prime(t_1)$, and then showing that $D(t)$ satisfies the stated conditions.

The final lemma we need requires a slight change from what is proven in \cite[Lemma 12.14]{BEE}, as it involves the $f$-generic condition.  Before stating the lemma, we need a definition.

\begin{defn} Let $c$ be a complete timelike geodesic such that $\Ric_f^m(c^\prime,c^\prime)\geq 0$ and $R_f(t_1)\not\equiv 0$.  Define
\[ \mathcal{A} = \{A\colon A\text { is a Lagrange tensor along $c$ with } A(t_1)=E\}. \]
Then we can divide $\mathcal{A}$ into two sets,
\[ L_+ = \{A\in\mathcal{A}\colon \theta_f(t_1)=\tr(A^\prime(t_1))\geq 0 \} \]
and
\[ L_- = \{A\in\mathcal{A}\colon \theta_f(t_1)=\tr(A^\prime(t_1))\leq 0 \} \]
\end{defn}

\begin{lem} \label{lemma4} Let $c$ be a complete timelike geodesic such that $\Ric_f^m(c^\prime,c^\prime)\geq 0$ and $R_f(t_1)\not\equiv 0$ \textup{(}if $m=\infty$, assume further that $f\leq k$\textup{)}.  Then each $A\in L_-$ satisfies $\det A(t)=0$ for some $t>t_1$, and each $A\in L_+$ satisfies $\det A(t)=0$ for some $t<t_1$. \end{lem}

\begin{proof}

If $A\in L_+$, then $\theta_f(t_1)\geq 0$.  Using the vorticity-free $(m,f)$-Raychaudhuri equation for timelike geodesics with $\Ric_f^m(c^\prime,c^\prime)\geq 0$ and $\tr(\sigma_f^2)\geq 0$, we find $\theta_f^\prime(t)\leq 0$ for all $t$.  Hence $\theta_f(t)\geq 0$ for all $t\leq t_1$.  If $\theta_f(t_0)>0$ for some $t_0\leq t_1$, then the result follows from Proposition \ref{raych} or Proposition \ref{raych2}, depending on the value of $m$.  Thus assume that $\theta_f(t)=0$ for $t\leq t_1$.  Then $\theta_f^\prime(t)$=0 for $t\leq t_1$, and so $\tr(\sigma_f^2)=0$.  Hence $\sigma_f=0$ for $t\leq t_1$, by Remark \ref{selfadjoint}.  Since $\theta_f=0$ and $B_f$ is self adjoint, we then have that $B_f=\theta_f=0$ for $t\leq t_1$.  This implies that $R_f=-B_f^2-B_f^\prime-\frac{2}{n-1}(f\circ c)^\prime B_f = 0$ for $t\leq t_1$, contradicting our assumption that $R_f(t_1)\not= 0$.

If $A\in L_-$, the proof is similar.
\end{proof}

\begin{proof}[Proof of Proposition \ref{conjpts}]

Let $c\colon\mathbb{R}\to M$ be a complete timelike geodesic with $\Ric_f^m(c^\prime,c^\prime)\geq 0$ and $R_f(t_1)\not= 0$ for some $t_1\in\mathbb{R}$.  Suppose that $c$ has no conjugate points.  Then let $D$ be the Lagrange tensor on $N(c(t))$ with $D(t_1)=E$ as constructed in Lemma \ref{lemma3}.  Since $c\mid_{[t_1,\infty)}$ has no conjugate points, $D(t)$ is nonsingular for all $t\geq t_1$.  Thus $D\not\in L_-$ by Lemma \ref{lemma4}.  Hence $D\in L_+$ and, moreover, $\tr (D^\prime(t_1)) > 0$ as $D\not\in L_-$.  Since $D^\prime(t_1) = \lim D_s^\prime(t_1)$, there is an $s>t_1$ such that $\tr (D_s^\prime(t_1)) > 0$.  Hence by Lemma \ref{lemma4}, there exists a $t_2<t_1$ and a nonzero tangent vector $v\in N(c(t_1))$ such that $D_s(t_2)(v)=0$.  From the proof of Lemma \ref{lemma2}, we have that $D_s(s)=0$ but $D_s^\prime(s) = -(A^\ast)^{-1}(s)$ is nonsingular.  Thus, if we let $Y\in N(c)$ be the unique parallel vector field along $c$ with $Y(t_2)=v$, then $J=D_s(Y)$ is a nontrivial Jacobi field along $c$ with $Y(t_2)=Y(s)=0$, against contradicting that $c$ has no conjugate points.
\end{proof}

This then yields immediately the corollary.

\begin{cor} Let $(M,g)$ be a space-time of dimension $n\geq 2$, let $f$ be a smooth function on $M$, and let $m$ be a positive integer.  Suppose satisfies $\Ric_f^m\geq 0$ for timelike vectors and the $f$-generic condition \textup{(}if $m=\infty$, assume further that $f\leq k$\textup{)}.  Then each timelike geodesic of $(M,g)$ is either incomplete or else has a pair of conjugate points. \end{cor}

We now want to consider the analogous situation of conjugate points on null geodesics.  All the methods are exactly the same, except now we must assume $n=\dim M\geq 3$ because the dimension of the quotient bundle $G(\beta)$ to a null geodesic $\beta$ has $\dim G(\beta)=n-2$.  Thus we arrive at the analog of Proposition \ref{conjpts}:

\begin{prop} Let $f$ be a smooth function on $M$ and let $m$ be a positive integer.  Let $\beta\colon\mathbb{R}\to(M,g)$ be a complete null geodesic with $\Ric_f^m(\beta^\prime,\beta^\prime)\geq 0$ \textup{(}if $m=\infty$, assume further that $f\leq k$\textup{)}.  If $\dim M\geq 3$ and if $\overline{R_f}(t_1)\not=0$ for some $t_1\in\mathbb{R}$, then $\beta$ has a pair of conjugate points. \end{prop}

Combining this with Proposition \ref{conjpts}, we obtain the following theorem about the existence of conjugate points, which is used heavily in the singularity theorems.

\begin{thm} \label{nonspace_conjpts} Let $(M,g)$ be a space-time with $\dim M\geq 3$, let $f$ be a smooth function on $M$, and let $m$ be a positive integer.  If $\Ric_f^m\geq 0$ for all nonspacelike vectors \textup{(}if $m=\infty$, assume further that $f\leq k$\textup{)} and $M$ satisfies the $f$-generic condition, then each nonspacelike geodesic in $M$ is either incomplete or else has a pair of conjugate points. \end{thm}

\section{Singularity Theorems}

A key idea in many singularity theorems is to prove that certain conditions force the existence of nonspacelike lines and the existence of conjugate points along complete nonspacelike geodesics.  Thus the nonspacelike line cannot be complete, and so the space-time is not complete.  This idea is easily shown in the following proposition.

\begin{prop} Let $M$ be a chronological space-time of dimension $n\geq 3$ which satisfies the $f$-generic condition and the $(m,f)$-timelike convergence condition \textup{(}if $m=\infty$, assume additionally that $f\leq k$\textup{)}.  Then $M$ is either strongly causal or null incomplete. \end{prop}

This is an immediate consequence of the idea above for proving singularity theorems and the following proposition of \cite{BEE}.

\begin{prop}[{\cite[Proposition 12.39]{BEE}}] If $M$ is a chronological space-time such that each inextendible null geodesic has a pair of conjugate points, then $M$ is strongly causal. \end{prop}

We can strengthen this singularity proposition by introducing the notion of causal disconnectedness, which is in some sense analogous to the idea of a Riemannian manifold having multiple ends.

\begin{defn} A space-time $M$ is causally disconnected if there exists a compact set $K\subset M$ and sequences $\{p_n\}$, $\{q_n\}$ such that $p_n << q_n$, $p_n,q_n\to\infty$, and every causal curve from $p_n$ to $q_n$ intersects $K$. \end{defn}

It is well-known that in a stably causal space-time which is causally disconnected, one can construct nonspacelike geodesic lines, which leads to our first singularity theorem.

\begin{thm} \label{sing1} Let $M$ be a chronological space-time of dimension $n\geq 3$ which is causally disconnected.  If $M$ satisfies the $f$-generic condition and the $(m,f)$-timelike convergence condition \textup{(}if $m=\infty$, assume additionally that $f\leq k$\textup{)}, then $M$ is nonspacelike incomplete. \end{thm}

\begin{proof}

Assume that all nonspacelike geodesics of $(M,g)$ are complete.  Then we know that every nonspacelike geodesic has conjugate points, and so the space-time is strongly causal.  On the other hand, since the space-time is causally disconnected, it must contain an inextendible maximal nonspacelike geodesic, which necessarily has no conjugate points, a contradiction.
\end{proof}

This idea of using causal disconnectedness to construct nonspacelike lines, and thus contradict the existence of conjugate points under appropriate curvature conditions is an underlying theme in many singularity theorems.  For this reason, we wish to find conditions which imply a space-time is causally disconnected.  This is the main theorem of Hawking and Penrose \cite[p. 538]{HP}, and to state it, we first need a definition.

\begin{defn} An achronal set $S\subset M$ is future trapped [resp., past trapped] if $E^+(S)$ [resp., $E^-(S)$] is compact. \end{defn}

By proving that the existence of a future (or past) trapped set in $M$ implies that $M$ is causally disconnected, Hawking and Penrose proved the following theorem.

\begin{thm}[Hawking-Penrose Singularity Theorem] No space-time $M$ of dimension $n\geq 3$ can satisfy all of the following three requirements:
\begin{itemize}
\item[(1)] $M$ is chronological.
\item[(2)] Every inextendible nonspacelike geodesic in $M$ contains a pair of conjugate points.
\item[(3)] There exists a future trapped or past trapped set $S\subset M$.
\end{itemize}
\end{thm}

Using an analogous argument to Theorem \ref{sing1}, this implies a theorem which is more similar to Theorem \ref{sing1}.

\begin{thm} \label{sing2} Let $M$ be a chronological space-time of dimension $n\geq 3$ which satisfies the $f$-generic condition and the $(m,f)$-timelike convergence condition \textup{(}if $m=\infty$, assume additionally that $f\leq k$\textup{)}.  If $M$ contains a trapped set, then $M$ is nonspacelike incomplete. \end{thm}

We now turn to proving Theorem \ref{singularity}.  To do so, we first define closed trapped surfaces.

\begin{defn} A closed trapped surface in a space-time $M$ is a compact spacelike submanifold of codimension $2$ for which the trace of both null second fundamental forms is either always positive or always negative. \end{defn}

The key to proving that the first condition of Theorem \ref{singularity} implies nonspacelike incompleteness is the following proposition.

\begin{prop} Let $M$ be a strongly causal space-time of dimension $n\geq 3$ which satisfies the condition $\Ric_f^m(v,v)\geq 0$ for all null vectors $v\in TM$ \textup{(}if $m=\infty$, assume additionally that $f\leq k$\textup{)}.  If $M$ contains a closed trapped surface $H$, then $M$ contains a trapped set or $M$ is null incomplete. \end{prop}

Together with Theorem \ref{sing2}, this gives us that the first condition of Theorem \ref{singularity} implies nonspacelike incompleteness.  The proof of this proposition is contained in the proof of \cite[Proposition 12.45]{BEE}, except that it is proven for the Ricci curvature.  To overcome this, we need the following lemma,

\begin{lem} Let $M$ be a space-time of dimension $n\geq 3$ and let $H$ be a spacelike submanifold of dimension $n-2$.  Suppose that $\beta\colon J\to M$ is an inextendible null geodesic which is orthogonal to $H$ at $p=\beta(t_1)$ and satisfies $\Ric_f^m(\beta^\prime,\beta^\prime)\geq 0$ \textup{(}if $m=\infty$, assume additionally that $f\leq k$\textup{)}.  Let $L=L_{\beta^\prime(t_1)}$ be the second fundamental form of $H$ with respect to $\beta^\prime(t_1)$.  If $\tr L+\lp\nabla f,\beta^\prime\rp$ has positive \textup{(}resp., negative\textup{)} value $\theta_1$ at $p$, then there is a focal point $t_0\in[t_1-(n-2)/\theta_1,t_1]$ \textup{(}resp., $t_0\in[t_1,t_1-(n-2)/\theta_1]$\textup{)} to $H$ along $\beta$, provided that $t_0\in J$. \end{lem}

This is a consequence of the analogues of Proposition \ref{raych} and \ref{raych2} for a Lagrange tensor $\overline{A}$ along a null geodesic $\beta$, by making the observation that $-(\tr L+\lp\nabla f,\beta^\prime\rp)$ is the expansion for a certain Lagrange tensor along $\beta$.  For details of how this is carried out, see \cite[Proposition 12.22]{BEE}, which gives the argument in the timelike case.

We now turn our attention to proving that the second condition of Theorem \ref{singularity} implies nonspacelike incompleteness.  First, we define what it means for a null geodesic to be reconverging.

\begin{defn} Let $M$ be a space-time and let $p\in M$, and let $\beta\colon [0,b)\to M$ be a null geodesic with $\beta(0)=p$.  Let $\overline{A}$ be the Lagrange tensor field along $\beta$ with $\overline{A}(0)=0$ and $\overline{A}^\prime(0)=E$.  Then $\beta$ is said to be $f$-reconverging in the future (resp., past) if the expansion $\overline{\theta_f}$ of $\overline{A}$ becomes negative (resp., positive) for some $t\in[0,b)$. \end{defn}

Thus, if we let $\beta$ be a future complete null geodesic starting at $p$ which is reconverging in the future, the analogues of Proposition \ref{raych} and \ref{raych2} imply that there is a $t>0$ such that $\beta(0)$ and $\beta(t)$ are conjugate, so $\beta(\tau)\in I^+(p)$ for $\tau>t$.  Since the null directions from a point form a compact set, this shows that if all null geodesics starting at $p$ are reconverging in the future of $p$, then $E^+(p)$ is compact, and so $\{p\}$ is a trapped set in $p$.  Together with Theorem \ref{sing2}, this gives that the second condition in Theorem \ref{singularity} implies nonspacelike incompleteness.

Finally, if $M$ contains a compact spacelike hypersurface $S$ which is achronal, then $E^+(S)=S$, and so $S$ is a trapped set in $M$.  On the other hand, if $S$ is not achronal, we can construct a covering manifold $\widetilde{M}$ of $M$ which contains a compact achronal spacelike hypersurface $\widetilde{S}$ (for details, see \cite{O}).  Applying Theorem \ref{sing2} to $\widetilde{M}$, we get that $\widetilde{M}$ is nonspacelike incomplete with the pullback metric, and so $M$ is nonspacelike incomplete as well, which shows that the last condition of Theorem \ref{singularity} implies nonspacelike incompleteness.

\section{The Maximum Principle}

As in the proof of the splitting theorem for manifolds satisfying the timelike convergence condition, Busemann functions play a critical role.  Here and in the proof of the splitting theorem, we will follow \cite{GH}, highlighting changes that must be made to deal with the $(m,f)$-timelike convergence condition.  The most important result we consider is the maximum principle for spacelike hypersurfaces, given in Theorem \ref{max_principle}.  Before proceeding, we remind the reader of a few definitions.

\begin{defn} Let $S$ be a subset of $M$.  A future-inextendible nonspacelike geodesic $\alpha\colon[0,a)\to M$ is a future $S$-ray if $d(S,\alpha(t))=t$ for all $t\in[0,a)$.  A ray is an $\alpha(0)$-ray.  Past-directed $S$-rays are defined similarly. \end{defn}

If we introduce an auxillary complete Riemannian metric $h$ on $M$, any timelike $S$-ray $\alpha$ can be defined on the interval $[0,\infty)$ by parametrization with respect to $h$-arclength.  Then, in an analogous, but more technical, manner, we can construct a generalized co-ray $\mu\colon[0,\infty)\to M$ such that $\mu(0)=p$ and $\mu^\prime(0)$ is the limit of some sequence of timelike curves $\mu_n$ with $\mu_n(0)=p$ and $\mu_n(t)=\alpha(n)$ for some $t$, which exist for $n$ sufficiently large.  For the details of this construction, see \cite{GH} or \cite[Chapter 14]{BEE}.  In general, $\mu$ is only guarenteed to be nonspacelike, but we would like to know that $\mu$ is timelike for the splitting theorem, and so we make the next definition.

\begin{defn} $M$ satisfies the generalized timelike co-ray condition at $p\in M$ if every generalized co-ray to a timelike $S$-ray starting at $p$ is timelike. \end{defn}

This condition is useful in proving various regularity results for Busemann functions that will be necessary to establishing the splitting theorem.  On the other hand, it can be proven that, given a timelike $S$-ray $\gamma$, there is an open neighborhood $U$ of $\gamma$ on which the generalized co-ray condition holds.  Again, we refer the reader to \cite{BEE} and \cite{GH} for details.

A primary tool in proving the splitting theorem is the maximal principal along spacelike hypersurfaces.  We will need to establish a variant of this, using a generalization of the Laplacian.

\begin{defn} If $f$ is a smooth function on $M$, the $f$-Laplacian is defined by
\[ \Delta_f u = \Delta u - \lp\nabla f, \nabla u\rp \]
for all $u\in C^2(M)$, where $\Delta u = \tr (\Hess u)$ and $\nabla$ is the gradient.  \end{defn}

We can then get a bound on the $f$-Laplacian of the distance function in an analogous way to the bound on the regular Laplacian, and this will be necessary for the proof of the Lorentzian splitting theorem.  We establish bounds separately in the cases $m<\infty$ and $m=\infty$.

\begin{lem} \label{lapl_bound} Let $M$ be a space-time which satisfies the $(m,f)$-timelike convergence condition for $m<\infty$ and let $\alpha\colon[0,\infty)\to M$ be a ray.  Define $d_r(x)=d(x,\alpha(r))$.  Then
\[ \Delta_f d_r(q) \geq -\frac{n+m-1}{d(q,\alpha(r))}. \]
\end{lem}

\begin{proof}

Let $x\in I^-(\alpha(r))$ and let $c$ be a past-directed maximal geodesic from $\alpha(r)$ to $x$.  Calculating the derivative of $\Delta_f d_r\circ c$ and noting that $\nabla d_r = -c^\prime$ along $c$, we have
\begin{align*}
\left(\Delta_f d_r\circ c\right)^\prime & = \left(\Delta d_r\circ c\right)^\prime + \lp \nabla_{c^\prime} \nabla f, c^\prime\rp \\
& \geq \Ric(c^\prime,c^\prime) + \frac{1}{n-1}\left(\Delta d_r\circ c\right)^2 +\Hess f(c^\prime,c^\prime) \\
& = \Ric_f^m(c^\prime,c^\prime) + \frac{1}{n-1}\left(\Delta d_r\circ c\right)^2 + \frac{1}{m} \lp\nabla f, c^\prime\rp^2 \\
& \geq \Ric_f^m(c^\prime,c^\prime) + \frac{1}{n+m-1}\left(\Delta_f d_r\circ c\right)^2,
\end{align*}
where the first inquality comes from calculating the derivative of the Laplacian and applying the Schwarz inequality (cf. \cite[p. 536]{BEE}), and the second is the Schwarz inequality \eqref{schwarz}.  Now, using the $(m,f)$-timelike convergence condition and noting that $\Delta d_r\circ c\to -\infty$ as $t\to 0$, and thus so too does $\Delta_f d_r\circ c$, we can integrate from $0$ to $d(q,\alpha(r))$ to get the desired inequality
\[ \Delta_f d_r(q)\geq -\frac{n+m-1}{d(q,\alpha(r))}. \qedhere \]
\end{proof}

In the case where $m=\infty$, we can still get a bound on the $f$-Laplacian of the distance function, but it will depend on the function $f$.  The analogous bound in the Riemannian setting is given in \cite{FLZ}.

\begin{lem} \label{lapl_bound2}Let $M$ be a space-time which satisfies the $(m,f)$-timelike convergence condition and let $\alpha\colon[0,\infty)\to M$ be a ray.  Define $d_r(x)=d(x,\alpha(r))$.  Then for $q\in I^-(\alpha(r))$,
\[ \Delta_f d_r(q) \geq -\frac{n-1}{\rho} + \frac{2}{\rho}f(q) - \frac{2}{\rho^2}\int_0^\rho f(\sigma(t))\;dt, \]
where $\rho=d(q,\alpha(r))$, and $\sigma\colon[0,\rho]\to M$ is a past-directed maximal timelike geodesic from $\alpha(r)$ to $q$.
\end{lem}

\begin{proof}

The second variation formula applied to timelike geodesics is
\[ L^{\prime\prime}(0) = -\lp\sigma^\prime, \nabla_V V\rp\mid_0^\rho - \int_0^\rho \lp V^\prime, V^\prime\rp - \lp R(V,\sigma^\prime)\sigma^\prime, V\rp, \]
where $V$ is a variation vector field of $\sigma$ with $\lp V,\sigma^\prime\rp = 0$.  Let $E_1,\dotsc,E_{n-1},\sigma^\prime(0)$ be an orthonormal basis for $T_{\sigma(0)}M$, and extend to a neighborhood of $\sigma$ by parallel translation.  Define $X_i(t)=\frac{t}{\rho}E_i(t)$.  Letting $V=X_i$ and summing, the maximality of $\sigma$ and the assumption that $\Ric_f\geq 0$ along timelike vectors yields
\begin{align*}
-\Delta d_r(q) & \leq \frac{n-1}{\rho} - \int_0^\rho \frac{t^2}{\rho^2}\Ric(\sigma^\prime,\sigma^\prime)(t)\;dt \\
& \leq \frac{n-1}{\rho} + \int_0^\rho \frac{t^2}{\rho^2} \Hess f(\sigma^\prime,\sigma^\prime)(t)\; dt \\
& = \frac{n-1}{\rho} + (f\circ\sigma)^\prime(\rho) - \frac{2}{\rho^2}\int_0^\rho t (f\circ\sigma)^\prime(t)\;dt \\
& = \frac{n-1}{\rho} - \lp \nabla f, \nabla d_r\rp(q) - \frac{2}{\rho}f(q) + \frac{2}{\rho^2}\int_0^\rho f(\sigma(t))\;dt.
\end{align*}
Hence we arrive at the desired inequality
\[ \Delta_f d_r(q)\geq -\frac{n-1}{\rho} + \frac{2}{\rho}f(q) - \frac{2}{\rho^2}\int_0^\rho f\circ\sigma. \qedhere \]
\end{proof}

To establish the maximum principle, we will need an estimate on the Hessian of the distance function $d_r$, given in \cite{GH}.

\begin{lem}[{\cite[Lemma 4.3]{GH}}] \label{hess_estimate} Let $M$ be a future timelike geodesically complete space-time and let $\gamma$ be a timelike $S$-ray.  Assume the generalized timelike co-ray condition holds at $p\in I^-(\gamma)\cap I^+(S)$.  Then there exists a neighborhood $U$ of $p$ and constants $t_0>0$ and $A>0$ such that for each $q\in U$ and for each timelike asymptote $\alpha$ from $q$, we have
\[ \Hess d_t(v,v)\geq -A\lp v^\perp, v^\perp\rp \]
for all $v\in T_qM$ and $t\geq t_0$, where $v^\perp$ is the projection onto the normal space of $\alpha^\prime(0)$. \end{lem}

We then need to establish a maximum principle for $\Delta_f$ analogous to \cite[Proposition 4.4]{GH}.

\begin{thm} \label{max_principle} Let $M$ be a future timelike geodesicially complete space-time which satisfies the $(m,f)$-timelike convergence condition \textup{(}if $m=\infty$, assume further that $f\leq k$\textup{)}, and let $\gamma$ be a timelike future $S$-ray.  Let $W\subset I^-(\gamma)\cap I^+(S)$ be an open set on which the generalized timelike co-ray condition holds.  Let $\Sigma\subset W$ be a connected smooth spacelike hypersurface with nonpositive $f$-mean curvature $H_{f,\Sigma}\leq 0$.  If the Busemann function $b=b_\gamma^+$ attains a minimum along $\Sigma$, then $b$ is constant along $\Sigma$. \end{thm}

Here we define $H_{f,\Sigma} = H_\Sigma - \lp \nabla f, N\rp$, where $N$ is the future pointing unit normal along $\Sigma$, and we are using the sign convention $H_\Sigma = \divergence N$.

\begin{proof}

As in the proof in \cite{GH}, it suffices to prove that if $q\in\Sigma$ is a minimum of $b$ with $b(q)=a$, then there is a neighborhood of $q$ on which $b$ is constant.  Let $U$ be a neighborhood of $q$ on which Lemma \ref{hess_estimate} holds.  Suppose that $b$ is not constant on any neighborhood of $q$, so that there is a small coordinate ball $B\subset\Sigma\cap U$ centered at $q$ such that $\partial B\not= \partial^0 B$, where
\[ \partial^0 B = \{ x\in\partial B\colon b(x)=a\}. \]
From Lemma \ref{hess_estimate}, there is a constant $C>0$ such that
\[ \Hess d_t(v,v)\geq -C \]
for all asymptotes at $x$, for all $x\in B$, for all $v\in T_x\Sigma$ with $\lp v,v\rp\leq 1$, and for all $t$ sufficiently large.

Note that $b>a$ on $\partial B\setminus \partial^0 B$.  We claim that by choosing $B$ sufficiently small, we can construct a smooth function $h$ on $\Sigma$ having the following properties:
\begin{itemize}
\item $h(q)=0$,
\item $|\nabla_\Sigma h| \leq 1$ on $B$, where $\nabla_\Sigma$ is the gradient operator on $(\Sigma, g\mid_\Sigma)$,
\item $\Delta_{f,\Sigma} h\leq -D$ on $B$, where $D$ is a positive constant and $\Delta_{f,\Sigma}$ is the induced $f$-Laplacian on $\Sigma$, and
\item $h>0$ on $\partial^0 B$.
\end{itemize}
To construct $h$, we define $h=1-e^{\alpha\phi}$ for $\alpha>0$ a sufficiently large constant and $\phi$ an appropriate function.  This works exactly as in \cite{EH}, because a simple calculation shows that
\[ \Delta_{f,\Sigma} h = -e^{\alpha\phi}\left(\alpha^2\|\nabla_\Sigma\phi\|^2 + \alpha\Delta_{f,\Sigma}\phi\right), \]
and so the extra term in the $f$-Laplacian is not important to the construction.

Now consider the function $u_\varepsilon = b+\varepsilon h$.  By the construction of $h$, $u_\varepsilon$ attains a minimum at some point in $p\in B\setminus\partial B$.

Let $\alpha\colon[0,\infty)\to M$ be a timelike asymptote to $\gamma$ at $p$.  Then for each $t>0$, the function
\[ b_{p,t}(x) = b(p)+t-d(x,\alpha(t)) \]
is a smooth upper support function for $b$ at $p$.  Then the function $u_{\varepsilon,t} = b_{p,t}+\varepsilon h$ is a smooth upper support function for $u_\varepsilon$ at $p$, which implies that $u_{\varepsilon,t}$ also has a minimum at $p$.  We will derive a contradiction by showing that $\Delta_{f,\Sigma} u_{\varepsilon, t}(p)$ is negative for $\varepsilon$ sufficiently small and $t$ sufficiently large.  To begin, we have
\[ \Delta_{f,\Sigma} u_{\varepsilon,t}(p) = \Delta_{f,\Sigma} b_{p,t}(p)+\varepsilon \Delta_{f,\Sigma} h(p). \]
A simple calculation shows that the formula relating $\Delta_{f,\Sigma}$ to $\Delta_f$ is
\[ \Delta_{f,\Sigma} b_{p,t} = \Delta_f b_{p,t} + H_{f,\Sigma}\lp \nabla b_{p,t}, N\rp + \Hess b_{p,t}(N,N), \]
where $N$ is the future-directed normal to $\Sigma$.  Now, by Lemma \ref{lapl_bound} or Lemma \ref{lapl_bound2}, depending on whether $m<\infty$ or $m=\infty$, and our definition of $b_{p,t}$, we have that
\[ \Delta_f b_{p,t}(p) \leq \frac{n+m-1}{t} \]
or
\[ \Delta_f b_{p,t}(p) \leq \frac{n-1}{t} - \frac{2}{t}f(p) + \frac{2}{t^2}\int_0^t f\circ\sigma, \]
where $\sigma$ is a unit-speed past-directed maximal timelike geodesic from $\alpha(t)$ to $p$.

Now, since $\nabla_\Sigma u_{\varepsilon,t}(p)=0$ and $\nabla b_{p,t}(p) = -\alpha^\prime(0)$, the relation between $\nabla$ and $\nabla_\Sigma$ implies that
\[ N = \lp N, \alpha^\prime(0)\rp^{-1} (-\alpha^\prime(0) + \varepsilon \nabla_\Sigma h). \]
Using our estimate on $\nabla_\Sigma h$, our Hessian estimate for $d_t$, and the reverse Schwarz inequality for timelike vectors, we then have
\[ \Hess b_{p,t}(N,N)\mid_p \leq C\varepsilon^2. \]
Thus, if $m<\infty$, the $f$-mean curvature assumption implies
\[ \Delta_{f,\Sigma} b_{p,t}(p)\leq \frac{n+m-1}{t}+C\varepsilon^2, \]
giving us the estimate
\[ \Delta_{f,\Sigma} u_{\varepsilon,t}(p) \leq \frac{n+m-1}{t} + C\varepsilon^2 - D\varepsilon. \]
Hence, if we choose $\varepsilon$ sufficiently small and $t$ sufficiently large, the right side becomes negative, yielding the desired inequality.  On the other hand, if $m=\infty$, then
\[ \Delta_{f,\Sigma} u_{\varepsilon,t}(p) \leq \frac{n-1}{t} - \frac{2f(p)}{t} + \frac{2}{t^2} \int_0^t f\circ\sigma + C\varepsilon^2 - D\varepsilon. \]
Thus, if $f\leq k$, then choosing $\varepsilon$ sufficiently small and $t$ sufficiently large, the right hand side will be negative, again yielding the desired inequality.
\end{proof}

This has the following immediate consequences which are of use in the proof of the splitting theorem.  They appear in \cite{GH} as Corollary 4.5 and Corollary 4.6, respectively.

\begin{cor} \label{gh_cor45} Let $M$ be a future timelike geodesically complete space-time which satisfies the $(m,f)$-timelike convergence condition, and let $\gamma$ be a timelike $S$-ray.  Let $W\subset I^-(\gamma)\cap I^+(\gamma)$ be an open set on which the generalized co-ray condition holds.  Let $\Sigma$ be a connected acausal smooth spacelike hypersurface in $W$ with nonpositive $f$-mean curvature, $H_{f,\Sigma}\leq 0$.  Suppose that $\Sigma$ and $\Sigma_c=\{b=c\}\cap W$ have a point in common and that $\Sigma\subset J^+(\Sigma_c,W)$.  Then $\Sigma\subset\Sigma_c$. \end{cor}

\begin{cor} \label{gh_cor46} Let $M$ be a future timelike geodesically complete space-time which obeys the $(m,f)$-timelike convergence condition, and let $\gamma$ be a timelike $S$-ray.  Let $W\subset I^-(\gamma)\cap I^+(S)$ be an open set on which the generlized timelike co-ray condition holds.  Let $\Sigma$ be a smooth spacelike hypersurface with nonpositive $f$-mean curvature whose closure is contained in $W$.  Assume that $\Sigma$ is acausal in $W$ and $\overline{\Sigma}$ is compact.  If $\edge(\Sigma)\subset\{b\geq c\}$, then $\Sigma\subset\{b\geq c\}$. \end{cor}

\section{Proof of the Splitting Theorem}

For completeness, and to make it more clear how the various results we have assembled piece together, we give a complete sketch of the splitting theorem.

Suppose that $\gamma$ is a future-directed unit speed timelike line, and let $-\gamma$ denote the curve $-\gamma(t)=\gamma(-t)$.  Define
\[ b_r^+(x) = r-d(x,\gamma(r)), \qquad b_r^-(x)=r-d(-\gamma(r),x), \]
and define $b^+$ and $b^-$ by letting $r\to \infty$.  Since $\gamma\mid_{[a,\infty)}$ and $-\gamma\mid_{[a,\infty)}$ are timelike rays for all $a\in\mathbb{R}$, the regularity of Busemann functions implies that $b^+$ and $b^-$ are defined and finite valued on $I(\gamma)=I^+(\gamma)\cap I^-(\gamma)$.  Moreover, all the regularity results of \cite{GH} and their time duals hold on a neighborhood of each point of $\gamma$.

Now, the triangle inequality implies that $b^+ + b^-\geq 0$ on $I(\gamma)$, with equality holding along $\gamma$.  Let $U$ be a neighborhood of $\gamma(0)$ such that all the preceding regularity properties hold for $b^{\pm}$ on $U$.  Look at the sets $S^\pm = \{ b^\pm = 0\}\cap U$.  By \cite[Proposition 4.2]{GH}, $S^\pm$ are partial Cauchy surfaces in $U$ which both pass through $\gamma(0)$ and with $S^-$ lying to the future of $S^+$.

Let $W$ be a small coordinate ball in $S^+$ centered at $\gamma(0)$ whose closure is also in $S^+$.  By applying the fundamental existence result of \cite[Theorem 4.1]{B}, we get a smooth spacelike hypersurface $\Sigma\subset U$ such that $H_{f,\Sigma}=0$, $\Sigma$ is acausal in $U$ with compact closure, $\edge(\Sigma)=\edge(W)$, and $\Sigma$ meets $\gamma$.

By applying Corollary \ref{gh_cor46}, and its time dual to $b^+$ and $b^-$, we have that
\[ \Sigma \subset \{b^+\geq 0\} \cap \{b^-\geq 0\}. \]
This implies that $\Sigma$ in fact meets $\gamma$ at $\gamma(0)$, and so corollary \ref{gh_cor45} then implies that
\begin{equation} \label{keyeqn} b^+ = b^- = 0 \mbox{ along } \Sigma. \end{equation}

From each point $x\in\Sigma$, there exist timelike asymptotes $\alpha_x^\pm\colon[0,\infty)\to M$ to $\pm\gamma$.  Let $\alpha_x\colon(-\infty,\infty)\to M$ be the (possibly broken) geodesic defined by
\[ \alpha_x(t) = \left\{\begin{array}{ll}
\alpha_x^-(-t), & t\leq 0, \\
\alpha_x^+(t), & t\geq 0. \end{array} \right. \]
Using \eqref{keyeqn}, it follows that $b^+(\alpha_x(t))=t$ and $b^-(\alpha_x(t)) = -t$, which in turn implies that $\alpha_x$ is a line.  A similarly simple argument shows that $\alpha_x^\pm$ are $\Sigma$-rays.  This implies that $\alpha_x^\pm$ are focal point free and that $\alpha_x$ meets $\Sigma$ orthogonally.  It can further be shown that these normal geodesics do not intersect.  For details, see \cite[Chapter 14]{BEE}.

Now, consider the normal exponential map $E\colon\mathbb{R}\times\Sigma\to M$ defined by $E(t,q)=\exp (tN_q)$, where $N$ is the future-directed unit normal field along $\Sigma$.  Then $E$ is injective and nonsingular, and hence is a diffeomorphism.  In particular, all the geodesics $\gamma_q(t)=E(t,q)$ are focal point free.

A simple calculation shows that the mean curvature function $H_f = H_{f,t}$ of the foliation $\Sigma_t = E(\{t\}\times\Sigma)$ obeys the following evolution equation
\begin{equation} \label{meancurv} \frac{\partial H_f}{\partial t} = -\Ric(N,N) - \Hess f(N,N) - |\nabla N|^2,  \end{equation}
where $N$ has been extended to be the future-directed unit normal field to the $\Sigma_t$'s.  In the case that $m<\infty$, this together with \eqref{schwarz} implies that
\begin{align*}
\frac{\partial H_f}{\partial t} & \leq -\Ric_f^m(N,N) - \frac{1}{n-1}H^2 - \frac{1}{m} \lp\nabla f, N\rp^2 \\
& \leq -\Ric_f^m(N,N) - \frac{1}{n+m-1} H_f^2.
\end{align*}
Note that we have equality if and only if $\nabla N$ is proportional to the identity map and $H = -\frac{n}{m+1}N(f)$.  Hence, unless $\nabla N$ vanishes identically, either the inequality is strict or $H_f$ is nonzero at some point, and so in either case $H_f$ is nonzero at some point.  By the same argument as in Proposition \ref{raych}, $H_f$ must blow up in finite time along some normal geodesic, violating the fact that all the normal geodesics are focal point free.  Thus $N$ is parallel, and hence $E$ is an isometry.

On the other hand, if $m=\infty$, then \eqref{meancurv} implies that
\begin{align*}
\frac{\partial H_f}{\partial t} & \leq -\frac{1}{n-1}(H_f+\lp\nabla f, N\rp)^2 \\
& \leq -\frac{H_f^2}{n-1} - \frac{2H_f}{n-1}\frac{\partial}{\partial t}(f\circ E).
\end{align*}
Again, unless $\nabla N$ vanishes identically, $H_f$ is not identically zero, and the argument of Proposition \ref{raych2} will imply that $H_{f}$ must blow up in finite time along some normal geodesic.  Hence $N$ is parallel in this case, as well.

Hence $M$ splits in the required way in a neighborhood of $\gamma$.  This splitting can then be continued to all of $M$ in a fairly straightforward manner; for details, see \cite{BEE}.  Moreover, along the splitting direction, we must have $\Ric(N,N)=0$, and so $\Hess f(N,N)=\frac{1}{m}(Nf)^2$.  Hence $f$ must be constant along $N$ (in the case that $m=\infty$, we must use that $f$ is bounded above), completing the proof. \qed

\section{Examples and Conclusion}
\label{example_section}

In \cite{FLZ}, the Cheeger-Gromoll splitting theorem is proven to hold for the curvature bound $\Ric_f\geq 0$ under the assumption that $f$ is bounded above, which we have proven as well.  The following example, analogous to \cite[Example 7.2]{WW}, shows that we need $f$ to be bounded above.

\begin{ex} Let $M=\mathbb{R}\times S^{n-1}$ with the metric
\[ g=-dt^2 + \sech^2 t h, \]
where $t$ is the coordinate on $\mathbb{R}$, and $h$ is the standard Riemannian metric on $S^{n-1}$ with constant sectional curvature one.  In other words, $M$ is $n$-dimensional de Sitter space, which is globally hyperbolic, contains a line, and $\Ric=(n-1)g$, so that $\Ric(v,v)=-(n-1)$ for all unit timelike vectors $v$.  Now, if we define $f\colon M\to\mathbb{R}$ by $f(t)=\sinh^2(Kt)$, then if $x$ is a unit vector in $TS^{n-1}$, we can calculate
\begin{align*}
\Hess f(\partial_t,\partial_t) & = 4K^2\cosh^2(Kt)-2K^2 \\
\Hess f(x,x) & = -2K\sinh(t)\cosh(t)\sinh(Kt)\cosh(Kt) \\
& \geq -K\cosh^2(t)\cosh^2(Kt)
\end{align*}
Up to rescaling, any timelike vector $v\in T_{(t,p)}M$ can be written as $v=\partial_t+\lambda x$ for some unit vector $x\in TS^{n-1}$ with $0\leq\lambda^2\cosh^2t\leq 1$, and a quick computation yields
\begin{align*}
\Ric_f(\partial_t,\partial_t) & \geq 2K^2-(n-1) \\
\Ric_f(v,v) & \geq 4K^2\cosh^2(Kt)-2K^2-K\cosh^2(Kt),
\end{align*}
and so, for sufficiently large $K$, $\Ric_f(v,v)\geq 0$ for timelike vectors in $TM$.  Of course, de Sitter space doesn't split, so indeed showing that the assumption of an upper bound on $f$ is necessary in the splitting theorem. \end{ex}

In the Riemannian case, this type of example is rather rare, because if $\Ric<0$ but $\Ric_f\geq 0$, then $\Hess f>0$, which means that the manifold must be diffeomorphic to $\mathbb{R}^n$.  However, in the Lorentzian case, because the curvature condition is only stated for timelike vectors, the assumption that $\Ric<0$ and $\Ric_f\geq 0$ only implies that $(f\circ c)^{\prime\prime}>0$ for all timelike geodesics $c$, which only implies that the manifold is topologically $\mathbb{R}\times N$, where $N$ is some codimension one manifold.  In fact, it is straightforward to modify the above example to a warped product of the form $\mathbb{R}\times_\phi N$ with $N$ any Riemannian Einstein manifold, using the same warping function $\phi$ and the same function $f$, with $K$ sufficiently large.

We note also that, if we took $N$ to be Einstein with Einstein constant $\lambda>0$, then $N$ would be compact, and so all of the hypotheses of the singularity theorem \ref{singularity}, considering $\Ric_f$, would hold except for $f$ being bounded above, and yet $M$ is nonspacelike complete. 

An interesting question we can ask is related to the rigidity of Theorem \ref{singularity}, using condition 3.  This would give us the following conjecture.

\begin{conj} Let $M$ be a globally hyperbolic space-time with a compact Cauchy surface $S$, and suppose that $M$ satisfies the $(m,f)$-timelike convergence condition \textup{(}if $m=\infty$, assume $|f|<k$\textup{)}.  Then if $M$ is timelike complete, $M$ splits isometrically as $\mathbb{R}\times S$, and $f$ is constant along $\mathbb{R}$. \end{conj}

In the case with $f$ constant and $m=0$, this is exactly Bartnik's conjecture \cite[Conjecture 2]{B2}.  Obviously, if the above is true, Bartnik's conjecture is true.  However, Bartnik's conjecture is still open.  One possible way of proving the conjecture would be to prove the existence of a timelike line in $M$ and then use Theorem \ref{splitting}.  In the case of $\Ric$, another possible approach to proving the conjecture is by proving the existence of a constant mean curvature Cauchy hypersurface $S^\prime$ in $M$, as Bartnik has proven in \cite{B2} that this implies the necessary splitting.  We can then ask if the existence of a constant $f$-mean curvature Cauchy hypersurface will yield the conjecture as stated above.

Another question we can ask deals with the question of how much weaker the $(m,f)$-timelike convergence condition is.  In other words, does there exist a triple $(M,g,f)$ such that $\Ric_f^m\geq 0$, but there is no metric $g_1$ such that $\Ric(g_1)\geq 0$.  The analogous question has been posed in \ref{WW} for the Riemannian setting, and to the best of our knowledge, is still unanswered.

\end{document}